\documentclass[12pt,reqno]{amsart}
\usepackage{amsmath, amssymb, amsthm, amsfonts} 
\usepackage{url}
\usepackage{mathtools}
\usepackage[breaklinks]{hyperref}
\usepackage{tikz}
\usetikzlibrary{decorations.markings}
\usepackage{array}
\newcommand{\Arg}{\operatorname{Arg}}

\renewcommand{\Re}{\operatorname{Re}}

\renewcommand{\(}{\left\(}
\renewcommand{\)}{\right\)}
\renewcommand{\[}{\left\[}
\renewcommand{\]}{\right\]}

\numberwithin{equation}{section}
 \theoremstyle{plain}
\newtheorem{theorem}{Theorem}[section]
\newtheorem{lemma}[theorem]{Lemma}
\newtheorem{remark}[]{Remark}

\newtheorem{corollary}[theorem]{Corollary}

   \makeatletter
\def\proof{\@ifnextchar[{\@oproof}{\@nproof}}
\def\@oproof[#1][#2]{\trivlist\item[\hskip\labelsep\textit{#2 Proof of\
#1.}~]\ignorespaces}
\def\@nproof{\trivlist\item[\hskip\labelsep\textit{Proof.}~]\ignorespaces}

\makeatother

\begin{document}
\title[An Extension of Ramanujan-Guinand Identity]{An Extension of Ramanujan-Guinand Identity for the Dedekind zeta function and a new formula for $\zeta_\mathbb{F}(2)$ and $\zeta_\mathbb{F}(2m+1)$}

\author{Diksha Rani Bansal}
\address{Diksha Rani\\ Department of Mathematics \\
Indian Institute of Technology Indore \\
Indore, Simrol, Madhya Pradesh 453552, India.} 
\email{dikshaba1233@gmail.com,  mscphd2207141001@iiti.ac.in}

\author{Bibekananda Maji}
\address{Bibekananda Maji\\ Department of Mathematics \\
Indian Institute of Technology Indore \\
Indore, Simrol, Madhya Pradesh 453552, India.} 
\email{bibek10iitb@gmail.com, bmaji@iiti.ac.in}

\thanks{2020 \textit{Mathematics Subject Classification.} Primary 11M06,  11R42; Secondary 33E20.\\
\textit{Keywords and phrases.} Riemann zeta function,  Ramanujan-Guinand identity,  Dedekind zeta function, Ramanujan-Koshliakov formula, Special values.}

\begin{abstract}
On page 253 of his Lost Notebook, Ramanujan recorded an intriguing identity relating a generalized divisor function and a modified Bessel function, which was later rediscovered by Guinand and is now known as the Ramanujan–Guinand identity. In this paper, we establish a number field analogue of this identity for the Dedekind zeta function. In particular, we recover Ramanujan-Guinand identity, Ramanujan-Koshliakov identity, and also obtain the transformation formula for the logarithm of the Dedekind eta function.
In 1986, Zagier obtained a formula for $\zeta_\mathbb{F}(2)$ for any number field $\mathbb{F}$. Quite surprisingly, as an application of our main theorem, we also obtain a new formula for $\zeta_\mathbb{F}(2)$ for arbitrary number field $\mathbb{F}$. Moreover, we derive an elegant identity for $\zeta_\mathbb{F}(2m+1)$ for \emph{any} positive integer $m$ and \emph{any} number field $\mathbb{F}$.

%
\end{abstract}

\maketitle

\section{{\bf Introduction}}
The Riemann zeta function, defined for complex numbers $s\in \mathbb{C}$ with $\Re(s)>1$, 
$
\zeta(s) = \sum_{n=1}^{\infty}\frac{1}{n^s},
$
is a fundamental object in analytic number theory. It admits a meromorphic continuation to the entire complex plane, with a simple pole at $s=1$, and satisfies the following functional equation:
\begin{align}\label{zeta functional eqn}
\Omega(s) = \Omega(1-s),
\end{align}
where $\Omega(s) = \pi^{-\frac{s}{2}}\Gamma\left(\frac{s}{2}\right)\zeta(s)$.
The special values of the Riemann zeta function have long been a subject of deep interest. 
In particular, Euler obtained the following evaluation at the positive even integers,
\begin{align}\label{Euler's formula for even zeta values}
\zeta(2m) = (-1)^{m+1}\frac{(2\pi)^{2m}B_{2m}}{2(2m)!}, \quad m \in \mathbb{N},
\end{align}
where $B_{2m}$ denotes the $2m$th Bernoulli number.
In contrast, no such simple closed formula is known for the odd zeta values. In this context,
Ramanujan discovered a remarkable identity relating $\zeta(2m+1)$ to rapidly convergent exponential series that are interchanged under the transformation $x \mapsto \frac{1}{x}$, stated for any complex $x$ with $\Re(x)>0$, and $k \in \mathbb{Z} \setminus \{0\}$, as
\begin{align}\label{Ramanujan id}
H_{2k+1}(x) + (-1)^{k+1} H_{2k+1}\left(\frac{1}{x}\right) = \pi^{k+1}\sum^{k+1}_{j=0}(-1)^{j-1}\frac{B_{2j}}{(2j)!}\frac{B_{2k + 2 - 2j}}{(2k + 2 - 2j)!}x^{k+1-2j},
\end{align}
where $$H_{2k+1}(x) = (4\pi x)^{-k} \left( \frac{1}{2}\zeta(2k+1) + \sum_{n=1}^\infty \sigma_{-(2k+1)}(n)e^{-2n\pi x} \right),$$ 
with the generalized divisor function $\sigma_{z}(n) = \sum_{d|n}d^{z},
z \in \mathbb{C}$. The above identity was also extended to number fields independently by the current authors \cite{BM} and Banerjee, Gupta and Kumar \cite{BGK23}. Interestingly, the identity \eqref{Ramanujan id} is equivalent to the transformation formula for the Eisenstein series over the full modular group $\mathrm{SL}_2(\mathbb{Z})$. Ramanujan discovered several other identities that connect zeta values with the generalized divisor function $\sigma_{z}(n)$. A notable example is a transformation formula on page 253 of his Lost Notebook \cite{Ramanujan}, that connects generalized divisor function with modified Bessel function. 
It was later rediscovered by Guinand \cite{Guinand} in 1955, and is now commonly known as the Ramanujan–Guinand identity. For $z \in \mathbb{C}$, it is given as
\begin{align}\label{RG id}
\mathcal{I}_{z}(x) &= \mathcal{I}_{-z}\left(\frac{1}{x}\right), \quad \text{for} \,\, \Re(x) > 0,
\end{align}
where
\begin{align*}
\mathcal{I}_{z}(x) &= 4\sqrt{x}\sum_{n=1}^{\infty}\sigma_{-z}(n)n^{\frac{z}{2}}K_{\frac{z}{2}}(2n\pi x) + \Omega(z)\left(x^{\frac{1-z}{2}} - x^{\frac{z-1}{2}}\right).
\end{align*}
Here, $\Omega(z)$ is defined in \eqref{zeta functional eqn} and $K_{\nu}(x)$ is the modified Bessel function of second kind. Various proofs of this formula appear in the literature. Berndt, Lee, and Sohn \cite{BLS} established \eqref{RG id} using a theorem of Watson \cite{Watson}, which essentially elaborates on Guinand’s original proof \cite{Guinand}. Dixit \cite{Dixit} later provided a proof of the Ramanujan–Guinand identity \eqref{RG id} via Mellin transform techniques.
The divisor function appearing in the Ramanujan–Guinand identity \eqref{RG id} also arises naturally in the Fourier expansion of the non-holomorphic Eisenstein series. 
We briefly recall this connection below.
For $\tau=x+iy$ with $y>0$, the non-holomorphic Eisenstein series $G(\tau,s)$ is defined by
\begin{align*}
G(\tau,s):= \frac{1}{2}\sum_{\substack{(m,n)\in\mathbb{Z}^2\\(m,n)\neq(0,0)}} \frac{y^{s}}{|m\tau+n|^{2s}}, \quad \Re(s)>1.
\end{align*}
This function is $\mathrm{SL}_2(\mathbb{Z})$-invariant and hence a non-holomorphic modular function of weight zero. It admits a meromorphic continuation to $\mathbb{C}$ with a simple pole at $s=1$ with residue $\pi/2$, given by the Fourier expansion
\begin{align}
G(\tau,s)
&= \zeta(2s)\,y^{s}
+ \frac{\sqrt{\pi}\,\Gamma\!\left(s-\tfrac12\right)}{\Gamma(s)}\zeta(2s-1)\,y^{1-s} \notag \\
&\quad + \frac{4\sqrt{y}\,\pi^{s}}{\Gamma(s)}
\sum_{n=1}^{\infty}\frac{\sigma_{2s-1}(n)}{n^{\,s-\frac12}}
K_{s-\frac12}(2\pi n y)\cos(2\pi n x). \label{Fourier series}
\end{align}
It is well known that \eqref{RG id} follows from \eqref{Fourier series} together with the modular invariance
\begin{align*}
G(\tau,s)=G\!\left(-\frac{1}{\tau},s\right), \qquad \tau\in\mathbb{H}.
\end{align*}
The identity \eqref{RG id} has been generalized and extended in several directions. Berndt, Dixit, and Sohn \cite{BDS} developed character analogues of the Ramanujan-Guinand identities. Cohen \cite{Cohen} presented a further generalization of \eqref{RG id}, for which Banerjee and Khyati \cite{BK24} later established character analogues. Banerjee and the second author \cite{BM23} investigated Cohen-type identities associated with generalized divisor functions. Dixit, Kesarwani, and Moll \cite{DKM18} derived another generalization while extending transformation formulas using generalized modified Bessel functions of the second kind. Subsequently, Kumar \cite{Kumar} generalized a theorem of Watson and, as an application, provided a new proof of the result of Dixit et al. \cite{DKM18}. A detailed historical discussion on the Ramanujan-Guinand formula \eqref{RG id} can be found in \cite{BLS}. In earlier developments, the case $z=0$ of \eqref{RG id}  was rediscovered by Koshliakov \cite{Koshliakov}, several years after Ramanujan’s original entry \cite[p.~253]{Ramanujan}. This special case, now known as the Ramanujan-Koshliakov identity, is equivalent to the functional equation for $\zeta^2(s)$ and is given by
\begin{align}\label{Koshliakov Formula}
\mathcal{I}_{0}(x) &= \mathcal{I}_{0}\left(\frac{1}{x}\right), \quad \text{for} \,\, \Re(x) > 0,
\end{align}
with
\begin{align*}
\mathcal{I}_{0}(x) = \sqrt{x}\left[\gamma - \log\left(\frac{4\pi}{x}\right)\right] + 4\sqrt{x}\sum_{n=1}^{\infty}d(n)\text{K}_0(2n\pi x),
\end{align*}
where $\gamma$ is the Euler-Mascheroni constant.
Thus the Ramanujan–Guinand identity \eqref{RG id} generalizes the above result in a natural and elegant way. The identity \eqref{Koshliakov Formula} was also proved by Ferrar \cite{Ferrar1936}. In his paper, Ferrar obtained solutions to the functional equation $F(x) = F\left(\frac{1}{x}\right)$. More recently, the present authors \cite{BM2} have derived solutions of this equation in the setting of number fields, namely for the functional equation of the higher powers of the Dedekind zeta function. The Dedekind zeta function is given by the Dirichlet series, valid for $\Re(s) > 1$,
\begin{align*}
\zeta_{\mathbb{F}}(s) = \sum_{n=1}^{\infty}\frac{\mathtt{a}_{\mathbb{F}}(n)}{n^s},
\end{align*}
where $\mathtt{a}_{\mathbb{F}}(n)$ is the number of ideals in the ring of integers $\mathcal{O}_{\mathbb{F}}$, with norm $n$. Throughout the paper, unless specified otherwise, let $\mathbb{F}$ be a number field of degree $d = r_1 + 2r_2$, with $r_1$ real embeddings and $r_2$ pairs of complex conjugate embeddings, and let $D$ denote its absolute discriminant.
The functional equation for the Dedekind zeta function is given by 
\begin{align}\label{Dedekind_zeta_functional_equation}
\Omega_{\mathbb{F}}(s) = \Omega_{\mathbb{F}}(1-s),
\end{align}
where 
\begin{align}\label{Omega_F of fnal eqn}
\Omega_{\mathbb{F}}(s) &= \left(\frac{D}{4^{r_2}\pi^d}\right)^{\frac{s}{2}}\Gamma^{r_1}\left(\frac{s}{2}\right)\Gamma^{r_2}(s)\zeta_{\mathbb{F}}(s).
\end{align}
We define the leading coefficient of the Laurent expansion of $\zeta_{\mathbb{F}}(s)$ at $s=0$ by
\begin{align}\label{Laurent series_at s=0_1st coeff}
C_{\mathbb{F}} := \lim_{s\rightarrow 0} \frac{\zeta_{\mathbb{F}}(s)}{s^{r_1+r_2-1}} = - \frac{R_\mathbb{F} h_\mathbb{F}}{w_\mathbb{F}},
\end{align}
where $h_\mathbb{F}$ is the class number of $\mathbb{F}$, $R_\mathbb{F}$ its regulator, and $w_\mathbb{F}$ the number of roots of unity in $\mathbb{F}$. It is also associated to the residue of the Dedekind zeta function at the pole $s=1$, given by
\begin{align}\label{Laurent series_at s=1_1st coeff}
\lim_{s\to1}(s-1)\zeta_{\mathbb{F}}(s) = \frac{2^{r_1}(2\pi)^{r_2}}{\sqrt{D}}\frac{R_\mathbb{F} h_\mathbb{F}}{w_\mathbb{F}} =: H_{\mathbb{F}}.
\end{align}
Note that $H_{\mathbb{F}}$ can be written as $\sqrt{D}H_{\mathbb{F}} = -2^{r_1}(2\pi)^{r_2}C_{\mathbb{F}}$.
A number field analogue of Euler’s identity \eqref{Euler's formula for even zeta values}
was established by Klingen \cite{Klingen} and Siegel \cite{Siegel}. They proved that for any totally real number field $\mathbb{F}$ of degree $d$,
\begin{align*}
\zeta_{\mathbb{F}}(2m) = \frac{q_m \pi^{2dm}}{\sqrt{D}}, \quad m \in \mathbb{N},
\end{align*}
where $q_m$ is a non-zero rational number.
Siegel also derived an explicit formula for $\zeta_{\mathbb{F}}(-1)$ for real quadratic fields, namely,
\begin{align}\label{Zeta(-1)}
\zeta_\mathbb{F}(-1) = \frac{1}{60}\sum_{\substack{x\in \mathbb{Z} \\ x^2 < D \\ x^2 \equiv D(4)}} \sigma_1\left(\frac{D-x^2}{4}\right),
\end{align}
which is directly related to $\zeta_{\mathbb{F}}(2)$ through the functional equation; for instance, $\zeta_{\mathbb{Q}(\sqrt{2})}(2) = \frac{\pi^4}{48\sqrt{2}}$.  A more precise evaluation is given in \cite{BCH21} for any real quadratic field $\mathbb{F}$:
\begin{align*}
\zeta_\mathbb{F}(2m) = \frac{\tau(\chi_D)\,(2\pi)^{4m}\, B_{2m}\, B_{2m,\chi_D}}
{4\,((2m)!)^{2}\, D^{2m}},
\end{align*}
where $\tau(\chi_D)$ is the Gauss sum attached to the quadratic character $\chi_D$, and $B_{2m,\chi_D}$ denotes the generalized Bernoulli numbers. 
In 1986, Zagier \cite{Zagier} in his seminal paper  obtained a formula for $\zeta_\mathbb{F}(2)$ for arbitrary number fields. Defining
\begin{align*}
A(x) = \int\limits_0^x \frac{1}{1+t^2}\log\frac{4}{1+t^2} \text{d}t,
\end{align*}
he showed that
\begin{align*}
\zeta_\mathbb{F}(2) = \frac{\pi^{2r_1+2r_2}}{\sqrt{D}}\sum_\nu c_\nu A(x_{\nu, 1})\cdots A(x_{\nu, r_2}),
\end{align*}
where the coefficients $c_\nu$ are rational numbers and the $x_{\nu,j}$ are real algebraic numbers. 

We often see that the transformation formulas for automorphic forms encode special values of zeta and $L$-functions. For example, transformation formula for Eisenstein series over full modular group is equivalent to Ramanujan's formula \eqref{Ramanujan id}  for odd zeta values. Likewise, the transformation formula for Hilbert modular forms leads to Siegel's evaluation of the special values $\zeta_\mathbb{F}(2m)$ over totally real number fields.
  The present work provides another instance of this philosophy. Starting from a number field analogue of the Ramanujan--Guinand identity which itself arises from the Fourier expansion of a non-holomorphic Eisenstein series and an associated transformation formula, we obtain a formula evaluation of $\zeta_\mathbb{F}(2)$ and $\zeta_\mathbb{F}(2m+1), m \in \mathbb{N}$, for an arbitrary number field. 
  
We next define a generalized divisor function associated with a number field $\mathbb{F}$, 
\begin{align}\label{Gen sigma fn}
\sigma_{\mathbb{F}, z}(n) := \sum_{d|n}\mathtt{a}_{\mathbb{F}}(d)\mathtt{a}_{\mathbb{F}}\left(\frac{n}{d}\right)d^z = n^{z}\sigma_{\mathbb{F}, -z}(n).
\end{align}
A trivial bound for $\sigma_{\mathbb{F},z}(n)$ is given by
\begin{align}\label{sigma bound}
\sigma_{\mathbb{F},z}(n) \ll n^{\Re(z)+\varepsilon},
\end{align}
for any $\varepsilon>0$. This follows from the estimate
$\mathtt{a}_{\mathbb{F}}(n) \ll n^{\varepsilon}$ due to
Chandrasekharan and Narasimhan \cite[Lemma~9]{CN1963}.
It can be verified that the Dirichlet series associated with $\sigma_{\mathbb{F},z}(n)$ is given by
\begin{align}\label{Dirichlet series for generalized divisor function}
\sum_{n=1}^{\infty}\frac{\sigma_{\mathbb{F}, z}(n)}{n^s} = \zeta_{\mathbb{F}}(s)\zeta_{\mathbb{F}}(s-z), \quad \Re(s) > \max\{1, 1+\Re(z)\}.
\end{align}
Prior to presenting the main result, we introduce the Steen function $V(x\mid a_1,a_2,\dots,a_n)$, defined as
\begin{align}\label{Steen fn}
V(x| a_1, a_2, \ldots, a_n) := \frac{1}{2\pi i}\int\limits_{(c)}\prod_{j=1}^{n}\Gamma(s + a_j)x^{-s}\text{d}s, \quad |\Arg(x)| < \frac{\pi n}{2}, 
\end{align}
where $(c)$ denotes the vertical line from $c - i\infty$ to $c + i\infty$ and $\Arg(x)$ is the principal argument $-\pi < \Arg(x) \leq \pi$. It is assumed that all poles of $\Gamma(s+a_j)$ lie on the left side of $(c)$.
This function is a special case of the Meijer $G$-function, \cite[p.~415, Def.~16.17]{OLBC2010}, given by
\begin{align*}
G^{n, 0}_{0, n}\left(x \left| \begin{matrix} \text{---} \\ a_1, \ldots, a_n \end{matrix} \right.\right) &= V(x| a_1, a_2, \ldots, a_n),
\end{align*} 
with various argument choices yielding some popular special functions. For example,
\begin{align} 
V(x| 0) &= e^{-x}, \,\,\, \text{if} \,\,\, c > 0 \,\,\, \text{and} \,\, \Re(x) > 0, \nonumber \\ V(x| a, b) &= 2x^{\frac{1}{2}(a+b)}K_{a-b}(2x^{\frac{1}{2}}), \,\,\, \text{if} \,\,\, c > \max\{-\Re(a), -\Re(b) \} \,\,\, \text{and} \,\, |\Arg(x)| < \pi. \label{Bessel fn formula} 
\end{align}
Further details on the Steen function can be found in \cite[p.~63]{KT} and \cite{Steen}.
To simplify notation, we shall use the following abbreviation throughout the paper:
\begin{align}\label{Short notation}
V_{\mathbb{F},z}(x) := V\left(x\bigg|\left(0\right)_{r_1 + r_2}, \left(\frac{z}{2}\right)_{r_1 + r_2}, \left(\frac{1}{2}\right)_{r_2}, \left(\frac{1+z}{2}\right)_{r_2}\right),
\end{align}
where $(a)_m$ is the $m$-tuple whose all entries are $a$, which is nothing but the following Meijer $G$-function:
\begin{align*}
V_{\mathbb{F},z}(x) = G^{2d, 0}_{0, 2d}\left(x \left| \begin{matrix} \text{---} \\ \left(0\right)_{r_1 + r_2}, \left(\frac{z}{2}\right)_{r_1 + r_2}, \left(\frac{1}{2}\right)_{r_2}, \left(\frac{1+z}{2}\right)_{r_2} \end{matrix} \right.\right).
\end{align*} 
This integral converges for $|\Arg(x)| < \pi d$, \cite[p.~415, Def.~16.17, Prop.~1]{OLBC2010}. We are now in a position to state the main results, which will be given in the next section. 

\section{{\bf Key Results}}
\begin{theorem}\label{Main thm 1}
Let $\mathbb{F}$ be any number field with degree $d$ and $z \in \mathbb{C}$ be any complex number. Let $C_{\mathbb{F}}$, $\Omega_{\mathbb{F}}(s), \sigma_{\mathbb{F}, z}(n)$ and $V_{\mathbb{F},z}\left(\left(\frac{\pi^{d} n x}{D}\right)^2\right)$ be defined as in \eqref{Laurent series_at s=0_1st coeff}, \eqref{Omega_F of fnal eqn}, \eqref{Gen sigma fn} and \eqref{Short notation}, respectively. Then, we have
\begin{align}\label{main eqn}
\mathcal{I}_{\mathbb{F}, z}(x) = \mathcal{I}_{\mathbb{F}, -z}\left(\frac{1}{x}\right), \quad \mathrm{for} \,\, x \neq 0,  |\Arg(x)| < \frac{\pi d}{2},
\end{align}
where $\mathcal{I}_{\mathbb{F}, z}(x)$ is defined as
\begin{align}
\mathcal{I}_{\mathbb{F}, z}(x) &= \left(\frac{\pi^d x}{D}\right)^{-\frac{z}{2}}2\sqrt{x}\sum_{n=1}^{\infty}\sigma_{\mathbb{F}, -z}(n)V_{\mathbb{F},z}\left(\left(\frac{\pi^{d} n x}{D}\right)^2\right) \notag\\
&+ 2^d\pi^{r_2}C_{\mathbb{F}}\Omega_{\mathbb{F}}(z)\left[x^{\frac{z-1}{2}} - x^{\frac{1-z}{2}}\right].\label{Defn of I_F}
\end{align}
\end{theorem}
\begin{remark}
We emphasize that the identity \eqref{main eqn} is established using the
functional equation of $\zeta_{\mathbb{F}}(s)\zeta_{\mathbb{F}}(s+z)$.
Interestingly, the transformation formula \eqref{main eqn} itself implies the functional equation of $\zeta_{\mathbb{F}}(s)\zeta_{\mathbb{F}}(s+z)$ shown in the following theorem.
\end{remark}

\begin{theorem}\label{Converse of main thm}
For any $z \in \mathbb{C}$, the relation \eqref{main eqn} leads to the following functional equation of $\zeta_{\mathbb{F}}(s)\zeta_{\mathbb{F}}(s+z)$:
\begin{align*}
\Omega_{\mathbb{F}}(s)\Omega_{\mathbb{F}}(s+z) &= \Omega_{\mathbb{F}}(1-s)\Omega_{\mathbb{F}}(1-s-z)
\end{align*}
where $\Omega_{\mathbb{F}}(s)$ is given in \eqref{Omega_F of fnal eqn}.
\end{theorem}

We now recover the classical Ramanujan–Guinand identity as a special case of our main result.
\begin{corollary}\label{Guinand from main thm}
When $\mathbb{F}=\mathbb{Q}$, the identity \eqref{main eqn} in Theorem~\ref{Main thm 1}
reduces to the Ramanujan–Guinand identity \eqref{RG id}.
\end{corollary}

We now turn to special cases of Theorem~\ref{Main thm 1} corresponding to particular values of $z$. The first such case arises when $z=0$.

\begin{theorem}\label{Case z=0 F}
For any number field $\mathbb{F}$, let $C_{\mathbb{F}}$ be defined as in \eqref{Laurent series_at s=0_1st coeff} and $a_1 =  \frac{\zeta^{(r_1+r_2)}_{\mathbb{F}}(0)}{(r_1+r_2)!}$. Then, we have
\begin{align*}
\mathcal{I}_{\mathbb{F}, 0}(x) = \mathcal{I}_{\mathbb{F}, 0}\left(\frac{1}{x}\right), \quad \mathrm{for} \,\, x \neq 0, |\Arg(x)| < \frac{\pi d}{2},
\end{align*}
where
\begin{align*}
\mathcal{I}_{\mathbb{F}, 0}(x) &= 2\sqrt{x}\sum_{n=1}^{\infty}\sigma_{\mathbb{F}, 0}(n)V_{\mathbb{F},0}\left(\left(\frac{\pi^{d} n x}{D}\right)^2\right) \notag \\
& + 4^{r_1+r_2}\pi^{r_2}C_{\mathbb{F}}\sqrt{x}\left[d\gamma C_\mathbb{F} - 2a_1 - C_\mathbb{F}\log\left(\frac{D}{4^{r_2}\pi^d x}\right)\right].
\end{align*}
\end{theorem}
It is worth noting that the classical Ramanujan-Koshliakov identity \eqref{Koshliakov Formula} is recovered as an immediate consequence of the above theorem for the field of rational numbers.

\begin{corollary}\label{Case Koshliakov id}
When $\mathbb{F}=\mathbb{Q}$, the identity obtained from Theorem~\ref{Case z=0 F} reduces to the Ramanujan-Koshliakov identity \eqref{Koshliakov Formula}.
\end{corollary}

Next, we consider the case $z=1$ in Theorem~\ref{Main thm 1}.
\begin{theorem}\label{Case z=1 F}
Let $V(x| a_1, a_2, \ldots, a_n)$ be the Steen function defined as in \eqref{Steen fn}. For any number field $\mathbb{F}$, the following identity holds,
\begin{align}\label{I_F(1,x)}
\mathcal{L}_{\mathbb{F}}(x) = \mathcal{L}_{\mathbb{F}}\left(\frac{1}{x}\right) - 2^{2r_1+2r_2}\pi^{r_2}C_{\mathbb{F}}^2 \log(x), \quad \mathrm{for}\,\, x \neq 0,  |\Arg(x)| < \frac{\pi d}{2},
\end{align}
where
\begin{align}
\mathcal{L}_{\mathbb{F}}(x) &=
2^{r_1+r_2}\sqrt{D}\sum_{n=1}^{\infty}\frac{\sigma_{\mathbb{F}, 1}(n)}{n}V\left(\left(\frac{(2\pi)^{d} n x}{D}\right)\bigg|\left(0\right)_{r_1 + r_2}, (1)_{r_2}\right) - \frac{2^{r_1}Dx}{\pi^{r_1+r_2}}C_{\mathbb{F}}\zeta_{\mathbb{F}}(2). \label{IFx}
\end{align}
\end{theorem}


When $\mathbb{F}=\mathbb{Q}$, Theorem \ref{Case z=1 F} reduces to the following well-known identity of Ramanujan.
\begin{corollary}\label{Case z=1 Q}
For $\Re(x) > 0$, we obtain
\begin{align}\label{Id z=1 Q}
\sum_{n=1}^{\infty}\frac{\sigma(n)}{n}e^{-2\pi nx} - \sum_{n=1}^{\infty}\frac{\sigma(n)}{n}e^{-\frac{2\pi n}{x}} = \frac{\pi}{12}\left(\frac{1}{x} - x \right) + \frac{1}{2}\log(x).
\end{align}
\end{corollary}
\begin{remark}
The identity in the above corollary is equivalent to the transformation formula for the logarithm of Dedekind’s eta function $\eta(z)$, namely,
\begin{align*}
\log\left(\eta\left(-\frac{1}{z}\right)\right) - \log(\eta(z)) &= \frac{1}{2}\log(-iz),
\end{align*}
upon setting $x=\frac{i}{z}$ in \eqref{Id z=1 Q}.
\end{remark}

\subsection{Special values of $\zeta_{\mathbb{F}}(s)$:} In this section, we show how the preceding theorems yield special values of $\zeta_{\mathbb{F}}(s)$. It is interesting to see that Theorem \ref{Case z=1 F} leads to a new formula for $\zeta_{\mathbb{F}}(2)$ for any number field $\mathbb{F}$.
\begin{theorem}\label{Dedekind at 2 thm}
For any number field $\mathbb{F}$ of degree $d$ and any $x \in \mathbb{C} \setminus \{0\}$ with $|\Arg(x)| < \frac{\pi d}{2}$, we have
\begin{align}\label{Dedekind zeta at 2}
\zeta_{\mathbb{F}}(2)\left(x-\frac{1}{x}\right) &= S_\mathbb{F}(x) - S_\mathbb{F}\left(\frac{1}{x}\right) - \frac{(2\pi)^{d}C_{\mathbb{F}}}{D}\log(x),
\end{align}
where
\begin{align*}
S_\mathbb{F}(x) := \frac{2^{r_2}\pi^{r_1+r_2}}{\sqrt{D}C_\mathbb{F}}\sum_{n=1}^{\infty}\frac{\sigma_{\mathbb{F}, 1}(n)}{n}V\left(\left(\frac{(2\pi)^{d} n x}{D}\right)\bigg|\left(0\right)_{r_1 + r_2}, (1)_{r_2}\right).
\end{align*}
\end{theorem}
\begin{remark}
We want to emphasize that the above result is true any non-zero complex $x$ with $|\Arg(x)| < \frac{\pi d}{2}$. In particular, taking $x=2$, we numerically evaluated $\zeta_{\mathbb{F}}(2)$ for some quadratic fields. See Table \ref{Zeta_2}.
\end{remark}
Next, we find a formula for $\zeta_{\mathbb{F}}(2m+1)$ using Theorem \ref{Main thm 1}.
\begin{theorem}\label{Dedekind zeta at 2m+1}
Let $m \in \mathbb{N}$. For any number field $\mathbb{F}$ of degree $d$ and any non-zero complex $x$ with $|\Arg(x)| < \frac{\pi d}{2}$, we have
\begin{align*}
& \zeta_{\mathbb{F}}(2m+1) \left[x^{m} - x^{-m}\right] \\
 &= \left( \frac{4^{r_2} \pi^d}{D}\right)^{m+\frac{1}{2}} \left( \frac{(m-1)!}{ 2^{1-2m} \sqrt{\pi} (2m-1)!} \right)^{r_1} \left(\frac{1}{(2m)!} \right)^{r_2}\left[K_{\mathbb{F}, m}(x) - K_{\mathbb{F}, m}\left(\frac{1}{x}\right)\right] \notag \\
&+ \sqrt{\frac{D}{4^{r_2}\pi^d}} \left(\frac{m! (m-1)!}{2^{1-2m} \sqrt{\pi} (2m-1)!} \right)^{r_1} (2m+1)^{r_2} \zeta_{\mathbb{F}}(2m+2) \left[x^{m+1} - x^{-m-1}\right],
\end{align*}
where
\begin{align*}
K_{\mathbb{F}, m}(x) = \left(\frac{D}{\pi^d}\right)^{m+\frac{1}{2}}\frac{x^m}{2^{d-1}\pi^{r_2}C_{\mathbb{F}}}\sum_{n=1}^{\infty}\sigma_{\mathbb{F}, -2m-1}(n)V_{\mathbb{F}, 2m+1}\left(\left(\frac{\pi^d n}{Dx}\right)^2\right).
\end{align*}
\end{theorem}

\begin{remark} 
If we consider $x$ to be a primitive $2(m+1)$-th root of unity such that $|\Arg(x)| < \frac{\pi d}{2}$, then we get
\begin{align*}
& \zeta_{\mathbb{F}}(2m+1) \left[x^{m} - x^{-m}\right] \\
 &= \left( \frac{4^{r_2} \pi^d}{D}\right)^{m+\frac{1}{2}} \left( \frac{(m-1)!}{ 2^{1-2m} \sqrt{\pi} (2m-1)!} \right)^{r_1} \left(\frac{1}{(2m)!} \right)^{r_2}\left[K_{\mathbb{F}, m}(x) - K_{\mathbb{F}, m}\left(\frac{1}{x}\right)\right].
\end{align*}
In particular, for $m=2$ and $x=e^{\frac{i\pi}{3}}$, we obtain a new formula for  $\zeta_{\mathbb{F}}(5)$ over any number field:
\begin{align*}
\zeta_{\mathbb{F}}(5) &= \frac{1}{i\sqrt{3}}\left( \frac{4^{r_2} \pi^d}{D}\right)^{\frac{5}{2}} \left( \frac{4}{3\sqrt{\pi}}\right)^{r_1} \left(\frac{1}{24} \right)^{r_2}\left[K_{\mathbb{F}, 2}\left(e^{\frac{i\pi}{3}}\right) - K_{\mathbb{F}, 2}\left(e^{-\frac{i\pi}{3}}\right)\right],
\end{align*}
where
\begin{align*}
K_{\mathbb{F}, 2}(x) := \left(\frac{D}{\pi^d}\right)^{\frac{5}{2}}\frac{x^2}{2^{d-1}\pi^{r_2}C_{\mathbb{F}}}\sum_{n=1}^{\infty}\sigma_{\mathbb{F}, -5}(n)V_{\mathbb{F}, 5}\left(\left(\frac{\pi^d n}{Dx}\right)^2\right).
\end{align*}
\end{remark}

\section{{\bf Prerequisites}}
In this section, we collect several lemmas and identities that will be further used
in the proofs of the main results. We begin with some standard properties of the gamma
function. The gamma function $\Gamma(s)$ satisfies the classical Legendre duplication formula and the functional equation given by
\begin{align}\label{Legendre's Duplication Formula}
\Gamma(s)\Gamma\left(s + \frac{1}{2}\right)  &= 2^{1-2s}\sqrt{\pi}\Gamma(2s),  \quad \Gamma(s+1) = s\Gamma(s).
\end{align}
We shall also require the following Laurent and Taylor series expansions. As $s\to1$, the
Dedekind zeta function admits the expansion,
\begin{align}
(s-1)\zeta_{\mathbb{F}}(s) &= H_{\mathbb{F}}+\gamma_{\mathbb{F}}(s-1)+O\!\left((s-1)^2\right),
\label{zeta around 1}
\end{align}
where $H_{\mathbb{F}}$ is the residue of the Dedekind zeta function at the pole $s=1$ given in \eqref{Laurent series_at s=1_1st coeff}. As $s\to 0$, we also have
\begin{align}
\frac{\zeta_{\mathbb{F}}(s)}{s^{r_1+r_2-1}}
&= C_{\mathbb{F}}+a_1 s+O(s^2), \label{zeta series around 0}
\end{align}
where $a_1 =  \frac{\zeta^{(r_1+r_2)}_{\mathbb{F}}(0)}{(r_1+r_2)!}$ and $C_\mathbb{F}$ is defined as in \eqref{Laurent series_at s=0_1st coeff}.
The gamma function has the following expansion as $s \to 0$,
\begin{align*}
\Gamma(s) &=\frac{1}{s}-\gamma+O(s),
\end{align*}
which yields
\begin{align}
(s\Gamma(s))^{m} &= 1-m\gamma s+O(s^2),\label{gamma power series} \\
\left(s\Gamma\!\left(\frac{s}{2}\right)\right)^{m} &= 2^{m}-2^{m-1}m\gamma s+O(s^2).
\label{gamma series s/2}
\end{align}
Finally, for $x>0$,
\begin{align}
x^{s} &= 1+s\log x+O(s^2).
\label{exp series}
\end{align}
Using the expansions
\eqref{zeta series around 0}, \eqref{gamma power series},
\eqref{gamma series s/2}, and \eqref{exp series}, we obtain the Laurent expansion of
$\Omega_{\mathbb{F}}(s)$ (defined in \eqref{Omega_F of fnal eqn}) about $s=0$:
\begin{align}\label{Omega series}
\Omega_{\mathbb{F}}(s) &= \frac{2^{r_1}C_{\mathbb{F}}}{s} + 2^{r_1}a_1 - 2^{r_1-1}d\gamma C_{\mathbb{F}} + 2^{r_1-1}C_{\mathbb{F}}\log\!\left(\frac{D}{4^{r_2}\pi^{d}}\right) + O(s).
\end{align}
The next lemma provides a Mellin–transform representation for a product of Gamma
functions that will play a key role in later sections.
\begin{lemma}\label{Lemma for Mellin trans}
The following identity holds for $\Re(s) > \max\{0, -\Re(z)\}$,
\begin{align*}
\left(\frac{D}{4^{r_2}\pi^{d}}\right)^{s}\Gamma^{r_1}\left(\frac{s}{2}\right)\Gamma^{r_1}\left(\frac{s+z}{2}\right)\Gamma^{r_2}(s)\Gamma^{r_2}(s+z)
&= \frac{2^{r_2 z + 1}}{(4\pi)^{r_2}}\int\limits_{0}^\infty x^{s-1}V_{\mathbb{F},z}\left(\left(\frac{\pi^{d} x}{D}\right)^2\right) \mathrm{d}x.
\end{align*}
\end{lemma}
\begin{proof}
From \eqref{Short notation} and the definition \eqref{Steen fn} of the Steen function, we have
\begin{align*}
V_{\mathbb{F},z}\left(x\right)
&= \frac{1}{2\pi i}\int\limits_{(c)}\left[\Gamma(s)\Gamma\left(s+\frac{z}{2}\right)\right]^{r_1+r_2}\!\left[\Gamma\left(s+\frac{1}{2}\right)\Gamma\left(s+\frac{1+z}{2}\right)\right]^{r_2}\!x^{-s} \textrm{d}s,
\end{align*}
where $c > \max\left\{0, -\frac{\Re(z)}{2}\right\}$.
We apply the duplication formula \eqref{Legendre's Duplication Formula} for the gamma function in the last equation and subsequently perform a change of variable $s \mapsto \frac{s}{2}$ to arrive at
\begin{align*}
V_{\mathbb{F},z}\left(x\right)
&= \frac{(4\pi)^{r_2}}{2^{r_2 z} 2\pi i}{\displaystyle\int\limits_{(c)}\!\left[\Gamma(s)\Gamma\left(s+\frac{z}{2}\right)\right]^{r_1}\!\left[\Gamma(2s)\Gamma(2s+z)\right]^{r_2}(4^{2r_2}x)^{-s}\textrm{d}s}\\
&= \frac{(4\pi)^{r_2}}{2^{r_2 z + 1} 2\pi i}{\displaystyle\int\limits_{(2c)}\!\Gamma^{r_1}\left(\frac{s}{2}\right)\Gamma^{r_1}\left(\frac{s+z}{2}\right)\Gamma^{r_2}(s)\Gamma^{r_2}(s+z)(4^{r_2}\sqrt{x})^{-s}\textrm{d}s}.
\end{align*}
We now substitute $x = \left(\frac{\pi^{d} x}{D}\right)^2$ to have
\begin{align}\label{Inverse Mellin}
V_{\mathbb{F},z}\left(\left(\frac{\pi^{d} x}{D}\right)^2\right)
&= \frac{(4\pi)^{r_2}}{2^{r_2 z + 1}2\pi i}{\displaystyle\int\limits_{(2c)}\!\left(\frac{D}{4^{r_2}\pi^{d}}\right)^{s}\Gamma^{r_1}\left(\frac{s}{2}\right)\Gamma^{r_1}\left(\frac{s+z}{2}\right)\Gamma^{r_2}(s)\Gamma^{r_2}(s+z)x^{-s}\textrm{d}s}.
\end{align}
Hence, the Mellin transform for the product of the gamma factors in the integrand above is given by
\begin{align*}
\left(\frac{D}{4^{r_2}\pi^{d}}\right)^{s}\Gamma^{r_1}\left(\frac{s}{2}\right)\Gamma^{r_1}\left(\frac{s+z}{2}\right)\Gamma^{r_2}(s)\Gamma^{r_2}(s+z)
&= \frac{2^{r_2 z + 1}}{(4\pi)^{r_2}}\int\limits_{0}^\infty x^{s-1}V_{\mathbb{F},z}\left(\left(\frac{\pi^{d} x}{D}\right)^2\right) \mathrm{d}x.
\end{align*}
\end{proof}
We have another important result.
\begin{lemma}\label{Lemma reln bw V_z and V_-z}
For $|\Arg(x)| < \frac{\pi d}{2}$, we have
\begin{align}
V_{\mathbb{F},z}\left(x^2\right)
= x^{z}V_{\mathbb{F},-z}\left(x^2\right), \label{V_z and V_-z}
\end{align}
where $V_{\mathbb{F},z}(x)$ is defined in \eqref{Short notation}.
\end{lemma}
\begin{proof}
For $c > \max\left\{0, -\frac{\Re(z)}{2}\right\}$ and $|\Arg(x)| < \frac{\pi d}{2}$, we consider
\begin{align*}
V_{\mathbb{F},z}\left(x^2\right)
= \frac{1}{2\pi i}\int\limits_{(c)}\left[\Gamma(s)\Gamma\left(s+\frac{z}{2}\right)\right]^{r_1+r_2}\!\left[\Gamma\left(s+\frac{1}{2}\right)\Gamma\left(s+\frac{1+z}{2}\right)\right]^{r_2}\!x^{-2s} \textrm{d}s.
\end{align*}
Changing the variable $s$ to $s-\frac{z}{2}$, we have
\begin{align*}
V_{\mathbb{F},z}\left(x^2\right) 
&= \frac{1}{2\pi i}\int\limits_{\left(c+\frac{z}{2}\right)}\left[\Gamma\left(s-\frac{z}{2}\right)\Gamma(s)\right]^{r_1+r_2}\!\left[\Gamma\left(s+\frac{1-z}{2}\right)\Gamma\left(s+\frac{1}{2}\right)\right]^{r_2}x^{-2s+z} \textrm{d}s \notag \\
&= x^{z}V_{\mathbb{F},-z}\left(x^2\right).
\end{align*}
\end{proof}
We now prove the main results of the paper in the next section.
\section{{\bf Proofs of Main Results}}
\begin{proof}[Theorem \rm{\ref{Main thm 1}}][]
Let
\begin{align}\label{Defn of S}
S_{\mathbb{F},z}(x) &:= \sum_{n=1}^{\infty}\sigma_{\mathbb{F}, -z}(n)V_{\mathbb{F},z}\left(\left(\frac{\pi^{d} n x}{D}\right)^2\right),  \quad \mathrm{for} \,\,\, |\Arg(x)| < \frac{\pi d}{2}.
\end{align}
We use the following bound for the Steen function:
\begin{align*}
V_{\mathbb{F},z}(x) \ll_d x^{\frac{1}{2d}\left(\frac{1-d\Re(z)}{2} - r_1 - r_2\right)}\exp{\left(-2dx^{\frac{1}{2d}}\right)},
\end{align*}
as $|x| \rightarrow \infty$, which follows from the asymptotic estimate of the Meijer $G$-function \cite[p.~180]{Luke}. 
Combining this with the bound \eqref{sigma bound} for $\sigma_{\mathbb{F},z}(n)$, one can see the absolute convergence of $S_{\mathbb{F},z}(x)$. Moreover, as $|x| \rightarrow \infty$, we have
\begin{align*}
S_{\mathbb{F},z}(x) \ll_{d, D} x^{\frac{1}{d}\left(\frac{1-d\Re(z)}{2} - r_1 - r_2\right)}\exp{\left(-2\pi d \left(\frac{x}{D}\right)^{\frac{1}{d}}\right)}.
\end{align*}
We use \eqref{Dirichlet series for generalized divisor function} and \eqref{Inverse Mellin} to rewrite $S_{\mathbb{F},z}(x)$ in the integral form as
\begin{align*}
S_{\mathbb{F},z}(x) = \frac{(4\pi)^{r_2}}{2^{r_2 z + 1}}\sum_{n=1}^{\infty}\frac{\sigma_{\mathbb{F}, -z}(n)}{2\pi i}{\displaystyle\int\limits_{(c)}\!\Gamma^{r_1}\left(\frac{s}{2}\right)\Gamma^{r_1}\left(\frac{s+z}{2}\right)\Gamma^{r_2}(s)\Gamma^{r_2}(s+z)\left(\frac{D}{4^{r_2}\pi^{d}nx}\right)^{s}\textrm{d}s} \notag \\
= \frac{(4\pi)^{r_2}}{2^{r_2 z + 1}2\pi i}\int\limits_{(c)}\zeta_{\mathbb{F}}(s)\zeta_{\mathbb{F}}(s+z)\Gamma^{r_1}\left(\frac{s}{2}\right)\Gamma^{r_1}\left(\frac{s+z}{2}\right)\Gamma^{r_2}(s)\Gamma^{r_2}(s+z)\left(\frac{D}{4^{r_2}\pi^{d} x}\right)^{s} \textrm{d}s.
\end{align*}
Assuming $c>\max\left\{1, 1-\Re(z)\right\}$, we justify the interchange of the series and the integral in the final step. We use \eqref{Dedekind_zeta_functional_equation} to arrive at
\begin{align}
S_{\mathbb{F},z}(x) &= \frac{(4\pi)^{r_2}}{2^{r_2 z + 1} 2\pi i}{\displaystyle\int\limits_{(c)}\left(\frac{D}{4^{r_2}\pi^d}\right)^{-\frac{z}{2}}\Omega_{\mathbb{F}}(s)\Omega_{\mathbb{F}}(s+z)x^{-s} \textrm{d}s} \nonumber \\
&= \left(\frac{\pi^d}{D}\right)^{\frac{z}{2}}\frac{(4\pi)^{r_2}}{2\times 2\pi i}\int_{(c)}\Omega_{\mathbb{F}}(s)\Omega_{\mathbb{F}}(s+z)x^{-s}\textrm{d}s. \label{S in integral}
\end{align}
We consider $\mathcal{C}$ to be the rectangular contour with vertices $c -iT, c+iT, \alpha + iT$, and $\alpha - iT$ traversed in the counter-clockwise direction. We choose $\alpha$ such that $\min\{-1, -1 \pm \Re(z)\} < \alpha < \min\{0, \pm \Re(z)\}$ and $c$ is assumed to be fixed as before. Therefore, by Cauchy's residue theorem, we have
\begin{align*}
\frac{1}{2\pi i}\int_{\mathcal{C}}\Omega_{\mathbb{F}}(s)\Omega_{\mathbb{F}}(s+z)x^{-s}\textrm{d}s &= \mathfrak{R}_{\mathbb{F}, z}(x), 
\end{align*}
where $\mathfrak{R}_{\mathbb{F}, z}(x)$ is the sum of residues of the integrand function inside the contour $\mathcal{C}$. As $T \rightarrow \infty$, using the Stirling's formula for $\Gamma(s)$, one can see that the horizontal integrals vanish. Thus, we are left with
\begin{align}
\frac{1}{2\pi i}\int_{(c)}\Omega_{\mathbb{F}}(s)\Omega_{\mathbb{F}}(s+z)x^{-s}\textrm{d}s &= \mathfrak{R}_{\mathbb{F}, z}(x) + \mathcal{J}_{\mathbb{F}, z}(x), \label{Contour integral 1}
\end{align}
where 
$$
\mathcal{J}_{\mathbb{F}, z}(x) := \frac{1}{2\pi i}\int_{(\alpha)}\Omega_{\mathbb{F}}(s)\Omega_{\mathbb{F}}(s+z)x^{-s}\textrm{d}s.
$$
To proceed further, we replace the variable $s$ by $1-s$ in $\mathcal{J}_{\mathbb{F}, z}(x)$, and then use the functional equation \eqref{Dedekind_zeta_functional_equation} for $\zeta_{\mathbb{F}}(s)$ and $\zeta_{\mathbb{F}}(s-z)$ to obtain
\begin{align*}
\mathcal{J}_{\mathbb{F}, z}(x) &= \frac{1}{2\pi i}\int_{(1-\alpha)}\Omega_{\mathbb{F}}(1-s)\Omega_{\mathbb{F}}(1-s+z)x^{-(1-s)}\textrm{d}s \\
&= \frac{1}{2\pi i x}\int_{(1-\alpha)}\Omega_{\mathbb{F}}(s)\Omega_{\mathbb{F}}(s-z)x^{s}\textrm{d}s.
\end{align*}
One can check that for any $z$, $1-\alpha >\max\left\{1, 1+\Re(z)\right\}$. Thus using \eqref{S in integral}, we can write the above integral as follows:
\begin{align*}
\mathcal{J}_{\mathbb{F}, z}(x) &= \frac{2}{x(4\pi)^{r_2}}\left(\frac{\pi^d}{D}\right)^{\frac{z}{2}}S_{\mathbb{F}, -z}\left(\frac{1}{x}\right).
\end{align*}
Substituting the above expression for $\mathcal{J}_{\mathbb{F},z}(x)$ into \eqref{Contour integral 1}, multiplying by $(4\pi)^{r_2}x^{\frac{1-z}{2}}$, and applying \eqref{S in integral}, we obtain
\begin{align}\label{Final eqn without R}
2\sqrt{x}\left(\frac{\pi^d x}{D}\right)^{-\frac{z}{2}}S_{\mathbb{F}, z}(x) = (4\pi)^{r_2}x^{\frac{1-z}{2}}\mathfrak{R}_{\mathbb{F}, z}(x) + \frac{2}{\sqrt{x}} \left(\frac{\pi^d}{Dx}\right)^{\frac{z}{2}}S_{\mathbb{F}, -z}\left(\frac{1}{x}\right).
\end{align}
We now turn our attention to the residual term $\mathfrak{R}_{\mathbb{F}, z}(x)$. Recall that $\zeta_{\mathbb{F}}(s)$ has a simple pole at $s=1$ and a zero of order $r_1+r_2-1$ at $s=0$. Moreover, $\zeta_{\mathbb{F}}(s)$ has
zeros of order $r_1 + r_2$ at negative even integers, and of order $r_2$ at negative odd integers. On the other hand, the gamma function $\Gamma(s)$ has simple poles on non-positive integers. Using these properties, we analyse that the points $s = 0$, $s=1, s= -z$ and $s=1-z$ are all simple poles of the integrand function inside the contour, when $z \neq 0, 1$.
Therefore, the term $\mathfrak{R}_{\mathbb{F}, z}(x)$ in \eqref{Final eqn without R} can be given as
\begin{align*}
\mathfrak{R}_{\mathbb{F}, z}(x) &= 
R_{0}(x,z) + R_{1}(x,z) + R_{1-z}(x,z) + R_{-z}(x,z).
\end{align*}
Here, $R_{a}(x,z)$ denotes the residue of the integrand $\Omega_{\mathbb{F}}(s)\Omega_{\mathbb{F}}(s+z)x^{-s}$ at simple pole $s=a$. First we calculate the residue corresponding to $s = 0$ as follows:
\begin{align*}
R_{0}(x,z) &= \lim_{s \rightarrow 0} s\,\Omega_{\mathbb{F}}(s)\Omega_{\mathbb{F}}(s+z)x^{-s} \\
&= \lim_{s \rightarrow 0} s\,\Omega_{\mathbb{F}}(s)\lim_{s \rightarrow 0}\Omega_{\mathbb{F}}(s+z)x^{-s} \\
&= \Omega_{\mathbb{F}}(z)\lim_{s \rightarrow 0}s^{r_1}\Gamma^{r_1}\left(\frac{s}{2}\right)s^{r_2}\Gamma^{r_2}(s)\frac{\zeta_{\mathbb{F}}(s)}{s^{r_1+r_2-1}} \\
&= 2^{r_1}C_{\mathbb{F}}\Omega_{\mathbb{F}}(z).
\end{align*}
where $C_\mathbb{F}$ is defined in \eqref{Laurent series_at s=0_1st coeff}. Moreover, all the remaining residual terms can be expressed in terms of $R_{0}(x,z)$. Therefore, we have
{\allowdisplaybreaks
\begin{align*}
R_{1}(x,z) &= \lim_{s \rightarrow 1} (s-1)\Omega_{\mathbb{F}}(s)\Omega_{\mathbb{F}}(s+z)x^{-s} \\
&= \lim_{s \rightarrow 1} (s-1)\Omega_{\mathbb{F}}(1-s)\Omega_{\mathbb{F}}(1-s-z)x^{-s}\\
&= -\lim_{s \rightarrow 0} s\,\Omega_{\mathbb{F}}(s)\Omega_{\mathbb{F}}(s-z)x^{s-1} \\
&= -\frac{1}{x}R_{0}\left(\frac{1}{x},-z\right), \\
R_{1-z}(x,z) &= \lim_{s \rightarrow 1-z}(s-1+z)\Omega_{\mathbb{F}}(s)\Omega_{\mathbb{F}}(s+z)x^{-s} \\
&= \lim_{s \rightarrow 1-z}(s-1+z)\Omega_{\mathbb{F}}(1-s)\Omega_{\mathbb{F}}(1-s-z)x^{-s}\\
&= -\lim_{s \rightarrow 0}s\,\Omega_{\mathbb{F}}(s+z)\Omega_{\mathbb{F}}(s)x^{s+z-1}\\
&= -x^{z-1}R_{0}\left(\frac{1}{x}, z\right), \\
R_{-z}(x,z) &= \lim_{s \rightarrow -z} (s+z)\Omega_{\mathbb{F}}(s)\Omega_{\mathbb{F}}(s+z)x^{-s} \\
&= \lim_{s \rightarrow 0} s\,\Omega_{\mathbb{F}}(s-z)\Omega_{\mathbb{F}}(s)x^{-s+z} \\
&= x^zR_{0}(x, -z).
\end{align*}}
Hence, the residual term $\mathfrak{R}_{\mathbb{F}, z}(x)$ becomes
\begin{align}
\mathfrak{R}_{\mathbb{F}, z}(x) &= R_{0}(x,z) + R_{-z}(x,z) + R_{1}(x,z) + R_{1-z}(x,z) \notag\\
&= R_{0}(x,z) + x^zR_{0}(x, -z) - \frac{1}{x}\left[R_{0}\left(\frac{1}{x},-z\right) + x^zR_{0}\left(\frac{1}{x}, z\right) \right] \notag\\
&= 2^{r_1}C_{\mathbb{F}}\Omega_{\mathbb{F}}(z) + x^z 2^{r_1}C_{\mathbb{F}}\Omega_{\mathbb{F}}(-z) - \frac{1}{x}\left[2^{r_1}C_{\mathbb{F}}\Omega_{\mathbb{F}}(-z) + x^z 2^{r_1}C_{\mathbb{F}}\Omega_{\mathbb{F}}(z) \right] \notag \\
&= 2^{r_1}C_{\mathbb{F}}\left[\Omega_{\mathbb{F}}(z)\left(1 - x^{z-1} \right) + \Omega_{\mathbb{F}}(-z)\left(x^{z} - x^{-1} \right)\right]. \label{Residual term}
\end{align}
Therefore, using \eqref{Residual term} in \eqref{Final eqn without R}, we finally have
\begin{align}
\left(\frac{\pi^d x}{D}\right)^{-\frac{z}{2}}&2\sqrt{x}S_{\mathbb{F}, z}(x) + 2^d\pi^{r_2}C_{\mathbb{F}}\Omega_{\mathbb{F}}(z)\left[x^{\frac{z-1}{2}} - x^{\frac{1-z}{2}} \right]  \notag \\
&= \frac{2}{\sqrt{x}}\left(\frac{\pi^d}{Dx}\right)^{\frac{z}{2}}S_{\mathbb{F}, -z}\left(\frac{1}{x}\right) + 2^d\pi^{r_2}C_{\mathbb{F}}\Omega_{\mathbb{F}}(-z)\left[x^{\frac{z+1}{2}} - x^{-\left(\frac{z+1}{2}\right)} \right]. \label{For S(1/x)}
\end{align}
Note that $z=0$ is a pole of the both sides of the above equation. Also, the left hand side of \eqref{For S(1/x)} has a removable singularity at $z=1$. Thus, by analytic continuation, we have the required identity \eqref{main eqn}, where we assume
\begin{align*}
\mathcal{I}_{\mathbb{F}, z}(x) = \left(\frac{\pi^d x}{D}\right)^{-\frac{z}{2}}2\sqrt{x}S_{\mathbb{F}, z}(x) + 2^d\pi^{r_2}C_{\mathbb{F}}\Omega_{\mathbb{F}}(z)\left[x^{\frac{z-1}{2}} - x^{\frac{1-z}{2}} \right].
\end{align*}
\end{proof}

\begin{proof}[Theorem \rm{\ref{Converse of main thm}}][]
We use Lemma~\ref{Lemma for Mellin trans} to write the following Mellin transform:
\begin{align*}
\left(\frac{D}{4^{r_2}\pi^{d}}\right)^{s}\Gamma^{r_1}\left(\frac{s}{2}\right)\Gamma^{r_1}\left(\frac{s+z}{2}\right)\Gamma^{r_2}(s)\Gamma^{r_2}(s+z)
&= \frac{2^{r_2 z + 1}}{(4\pi)^{r_2}}\int\limits_{0}^\infty t^{s-1}V_{\mathbb{F},z}\left(\left(\frac{\pi^{d} t}{D}\right)^2\right) \mathrm{d}x.
\end{align*}
Substituting $t = nx$, we get
\begin{align}\label{Change of variable for converse}
\left(\frac{D}{4^{r_2}\pi^{d}}\right)^{s}\Gamma^{r_1}\left(\frac{s}{2}\right)\Gamma^{r_1}\left(\frac{s+z}{2}\right)\Gamma^{r_2}(s)\Gamma^{r_2}(s+z)n^{-s} \notag \\
= \frac{2^{zr_2 + 1}}{(4\pi)^{r_2}}\int_0^\infty x^{s-1}V_{\mathbb{F},z}\left(\left(\frac{\pi^{d} nx}{D}\right)^2\right)\, \text{d}x.
\end{align}
We now consider $\Re(s) > \{1, 1-\Re(z)\}$. Multiplying \eqref{Change of variable for converse} by $\sigma_{\mathbb{F},-z}(n)\left(\frac{D}{4^{r_2}\pi^d}\right)^{\frac{z}{2}}$ and summing over $n$ to have $\Omega_{\mathbb{F}}(s)\Omega_{\mathbb{F}}(s+z)$ on the left hand side, we use the definition \eqref{Defn of S} of $S_{\mathbb{F},-z}(x)$ to arrive at
\begin{align}
\Omega_{\mathbb{F}}(s)\Omega_{\mathbb{F}}(s+z) &= \frac{2}{(4\pi)^{r_2}}\left(\frac{D}{\pi^d}\right)^{\frac{z}{2}}\int_0^\infty x^{s-1}\sum_{n=1}^{\infty}\sigma_{\mathbb{F},-z}(n)V_{\mathbb{F},z}\left(\left(\frac{\pi^{d} nx}{D}\right)^2\right)\, \text{d}x \notag\\
&= \frac{2}{(4\pi)^{r_2}}\left(\frac{D}{\pi^d}\right)^{\frac{z}{2}}\int_0^\infty x^{s-1}S_{\mathbb{F}, z}(x) \, \text{d}x \notag\\
&= \frac{2}{(4\pi)^{r_2}}\left(\frac{D}{\pi^d}\right)^{\frac{z}{2}}\left[\int_0^1 x^{s-1}S_{\mathbb{F}, z}(x) \text{d}x + \int_1^\infty x^{s-1}S_{\mathbb{F}, z}(x)\text{d}x \right] \notag\\
&= \frac{2}{(4\pi)^{r_2}}\left(\frac{D}{\pi^d}\right)^{\frac{z}{2}}\left[\int_1^\infty x^{-s-1}S_{\mathbb{F}, z}\left(\frac{1}{x}\right) \text{d}x + \int_1^\infty x^{s-1}S_{\mathbb{F}, z}(x) \text{d}x \right]. \label{Without S(1/x)}
\end{align}
We use \eqref{For S(1/x)} for the value of $S_{\mathbb{F}, -z}\left(\frac{1}{x} \right)$ and then replace $z$ by $-z$ to have
\begin{align*}
S_{\mathbb{F}, z}\left(\frac{1}{x} \right)
= \left(\frac{\pi^d}{D}\right)^{z}\!\!xS_{\mathbb{F}, -z}(x) + 2^{d-1}\pi^{r_2}\!C_{\mathbb{F}}\left(\frac{\pi^d}{D}\right)^{\frac{z}{2}}\bigg[\Omega_{\mathbb{F}}(-z)\left[x^{-z} - x \right] 
- \Omega_{\mathbb{F}}(z) \left[x^{1-z} - 1 \right]\bigg].
\end{align*}
Inserting the above value of $S_{\mathbb{F}, z}\left(\frac{1}{x} \right)$ in \eqref{Without S(1/x)}, we obtain
\begin{align}
\Omega_{\mathbb{F}}(s)&\Omega_{\mathbb{F}}(s+z) = 2^{r_1}C_{\mathbb{F}}\int_1^\infty x^{-s-1}\left\{\Omega_{\mathbb{F}}(-z)\left[x^{-z} - x \right] - \Omega_{\mathbb{F}}(z) \left[x^{1-z} - 1 \right]\right\}\text{d}x \nonumber \\
&+ \frac{2}{(4\pi)^{r_2}}\int_1^\infty \left\{ x^{-s}\left(\frac{D}{\pi^d}\right)^{-\frac{z}{2}}S_{\mathbb{F}, -z}(x) + x^{s-1}\left(\frac{D}{\pi^d}\right)^{\frac{z}{2}} S_{\mathbb{F}, z}(x) \right\} \text{d}x.
\label{Symmetry of funl eqn general for z}
\end{align}
It can be easily seen that the second integral in \eqref{Symmetry of funl eqn general for z} is symmetric under the substitutions $s \mapsto 1-s$ and $z \mapsto -z$. We denote the first integral in \eqref{Symmetry of funl eqn general for z} by 
\begin{align*}
K(s, z) &:= \int_1^\infty \Omega_{\mathbb{F}}(-z)\left[x^{-s-1-z} - x^{-s} \right] - \Omega_{\mathbb{F}}(z)\left[x^{-s-z} - x^{-s-1} \right] \text{d}x \\
&= \frac{\Omega_{\mathbb{F}}(-z)}{s+z} + \frac{\Omega_{\mathbb{F}}(-z)}{1-s} + \frac{\Omega_{\mathbb{F}}(z)}{1-s-z} + \frac{\Omega_{\mathbb{F}}(z)}{s} \\
&= K(1-s, -z).
\end{align*}
Therefore, using the above relation in \eqref{Symmetry of funl eqn general for z}, one arrives at the functional $\Omega_{\mathbb{F}}(s)\Omega_{\mathbb{F}}(s+z) = \Omega_{\mathbb{F}}(1-s)\Omega_{\mathbb{F}}(1-s-z)$. This completes the proof.
\end{proof}

\begin{proof}[Corollary \rm{\ref{Guinand from main thm}}][]
When $\mathbb{F} = \mathbb{Q}$, we have $D = 1, r_1 = 1, r_2 = 0$ and  $\sigma_{\mathbb{Q}, -z}(n) = \sigma_{-z}(n)$. We employ \eqref{Bessel fn formula} in \eqref{Defn of I_F} to have
\begin{align}
\mathcal{I}_{\mathbb{Q}, z}(x) &= (\pi x)^{-\frac{z}{2}}2\sqrt{x}\sum_{n=1}^{\infty}\sigma_{\mathbb{Q}, -z}(n)V\left((\pi n x)^2\bigg| 0, \frac{z}{2}\right) - \Omega_{\mathbb{Q}}(z)\left[x^{\frac{z-1}{2}} - x^{\frac{1-z}{2}}\right] \notag \\
&= 4(\pi x)^{-\frac{z}{2}}\sqrt{x}\sum_{n=1}^{\infty}\sigma_{-z}(n)(\pi n x)^{\frac{z}{2}}K_{\frac{z}{2}}(2\pi nx) + \Omega(z)\left[x^{\frac{1-z}{2}} - x^{\frac{z-1}{2}}\right]  \nonumber \\
&= 4\sqrt{x}\sum_{n=1}^{\infty}\sigma_{-z}(n)n^{\frac{z}{2}}K_{\frac{z}{2}}(2\pi nx) + \Omega(z)\left[x^{\frac{1-z}{2}} - x^{\frac{z-1}{2}}\right].\label{First sum_RG}
\end{align}
Using \eqref{First sum_RG} in \eqref{main eqn}, we get the Ramanujan-Guinand identity \eqref{RG id}.
\end{proof}

\begin{proof}[Corollary \rm{\ref{Case z=0 F}}][]
Putting $z=0$ in \eqref{main eqn}, we have
\begin{align*}
&\mathcal{I}_{\mathbb{F}, 0}(x) = \mathcal{I}_{\mathbb{F}, 0}\left(\frac{1}{x}\right).
\end{align*}
We take the series part of $\mathcal{I}_{\mathbb{F}, 0}(x)$ (as defined in 
\eqref{Defn of I_F}) and $\mathcal{I}_{\mathbb{F}, 0}\left(\frac{1}{x}\right)$ on one side of the above equation to have
\begin{align}
&2\sqrt{x}\sum_{n=1}^{\infty}\sigma_{\mathbb{F}, 0}(n)V_{\mathbb{F},0}\left(\left(\frac{\pi^{d} n x}{D}\right)^2\right) - \frac{2}{\sqrt{x}}\sum_{n=1}^{\infty}\sigma_{\mathbb{F}, 0}(n)V_{\mathbb{F},0}\left(\left(\frac{\pi^{d} x}{Dx}\right)^2\right)
 \notag \\
&= 2^d\pi^{r_2}C_{\mathbb{F}}\lim_{z \rightarrow 0}\Omega_{\mathbb{F}}(-z)\left[x^{\frac{z+1}{2}} - x^{\frac{-1-z}{2}}\right] - \Omega_{\mathbb{F}}(z)\left[x^{\frac{z-1}{2}} - x^{\frac{1-z}{2}}\right]. \label{IF0 series}
\end{align}
{\allowdisplaybreaks
To calculate the limit on the right hand side of \eqref{IF0 series}, we use the series expansions \eqref{exp series} and \eqref{Omega series} to have
\begin{align*}
&\lim_{z \rightarrow 0}\Omega_{\mathbb{F}}(-z)\left[x^{\frac{z+1}{2}} - x^{\frac{-1-z}{2}}\right] - \Omega_{\mathbb{F}}(z)\left[x^{\frac{z-1}{2}} - x^{\frac{1-z}{2}}\right] \notag \\
&= \lim_{z \to 0}\left\{-\frac{2^{r_1}C_\mathbb{F}}{z} + \left[2^{r_1}a_1 - 2^{r_1-1}d\gamma C_\mathbb{F} + 2^{r_1-1}C_\mathbb{F}\log\left(\frac{D}{4^{r_2}\pi^d}\right)\right] + O(z)\right\}\notag \\
&\times \left\{\left(\sqrt{x}-\frac{1}{\sqrt{x}}\right)+ \frac{z}{2}\log(x)\left(\frac{1}{\sqrt{x}}+\sqrt{x}\right)+ \cdots\right\}\notag \\
&- \left\{\frac{2^{r_1}C_\mathbb{F}}{z} + \left[2^{r_1}a_1 - 2^{r_1-1}d\gamma C_\mathbb{F} + 2^{r_1-1}C_\mathbb{F}\log\left(\frac{D}{4^{r_2}\pi^d}\right)\right] + O(z)\right\}\notag \\
&\times \left\{\left(\frac{1}{\sqrt{x}}-\sqrt{x}\right)+ \frac{z}{2}\log(x)\left(\frac{1}{\sqrt{x}}+\sqrt{x}\right)+\cdots \right\} \notag \\
&= 2\left(\sqrt{x}-\frac{1}{\sqrt{x}}\right)\left[2^{r_1}a_1 - 2^{r_1-1}d\gamma C_\mathbb{F} + 2^{r_1-1}C_\mathbb{F}\log\left(\frac{D}{4^{r_2}\pi^d}\right)\right] \notag\\
&- 2^{r_1}C_\mathbb{F}\log(x)\left(\frac{1}{\sqrt{x}}+\sqrt{x}\right) \notag \\
&= \frac{2^{r_1}}{\sqrt{x}}\left[d\gamma C_\mathbb{F} - 2a_1  - C_\mathbb{F}\log\left(\frac{Dx}{4^{r_2}\pi^d}\right)\right] - 2^{r_1}\sqrt{x}\left[d\gamma C_\mathbb{F} - 2a_1 - C_\mathbb{F}\log\left(\frac{D}{4^{r_2}\pi^d x}\right)\right].
\end{align*}}
Using the above value in \eqref{IF0 series}, we have the desired result.
\end{proof}

\begin{proof}[Corollary \rm{\ref{Case Koshliakov id}}][]
When $\mathbb{F} = \mathbb{Q}$, we have $D = 1, r_1 = 1, r_2 = 0, C_\mathbb{Q} = -\frac{1}{2}, a_1 =  \frac{\zeta^{(r_1+r_2)}_{\mathbb{F}}(0)}{(r_1+r_2)!} =  \zeta'(0) = -\frac{1}{2}\log(2\pi)$. So, we get
\begin{align*}
\mathcal{I}_{\mathbb{Q}, 0}(x) &= 2\sqrt{x}\sum_{n=1}^{\infty}\sigma_{\mathbb{Q}, 0}(n)V\left(\left(\pi n x\right)^2\bigg| 0, 0\right) - 2\sqrt{x}\left[-\frac{\gamma}{2} + \log(2\pi) + \frac{1}{2}\log\left(\frac{1}{\pi x}\right)\right] \notag \\
&= 4\sqrt{x}\sum_{n=1}^{\infty}d(n)K_0(2\pi n x) + \sqrt{x}\left[\gamma - \log\left(\frac{4\pi}{x}\right)\right].
\end{align*}
Substituting the above expression in Theorem \ref{Case z=0 F} gives the Ramanujan-Koshliakov identity \eqref{Koshliakov Formula}.
\end{proof}

\begin{proof}[Theorem \rm{\ref{Case z=1 F}}][]
Taking $z=1$ in Theorem \ref{Main thm 1}, we have
\begin{align*}
\mathcal{I}_{\mathbb{F}, 1}(x) = \mathcal{I}_{\mathbb{F}, -1}\left(\frac{1}{x}\right),
\end{align*}
where
\begin{align}
\mathcal{I}_{\mathbb{F}, 1}(x) &= 2\sqrt{\frac{D}{\pi^d}}\sum_{n=1}^{\infty}\sigma_{\mathbb{F}, -1}(n)V_{\mathbb{F}, 1}\left(\left(\frac{\pi^{d} n x}{D}\right)^2\right) + \lim_{z \rightarrow 1}2^d\pi^{r_2}C_{\mathbb{F}}\Omega_{\mathbb{F}}(z)\left[x^{\frac{z-1}{2}} - x^{\frac{1-z}{2}}\right]. \label{IF1}
\end{align}
We now calculate the above limit as follows:
{\allowdisplaybreaks
\begin{align}
&\lim_{z \rightarrow 1}2^d\pi^{r_2}C_{\mathbb{F}}\Omega_{\mathbb{F}}(z)\left[x^{\frac{z-1}{2}} - x^{\frac{1-z}{2}}\right] \notag \\
&= 2^d\pi^{r_2}C_{\mathbb{F}}\lim_{z \rightarrow 1}\left(\frac{D}{4^{r_2}\pi^d}\right)^{\frac{z}{2}}\Gamma^{r_1}\left(\frac{z}{2}\right)\Gamma^{r_2}(z)\zeta_{\mathbb{F}}(z)\left[x^{\frac{z-1}{2}} - x^{\frac{1-z}{2}}\right] \notag\\
&= 2^d\pi^{r_2}C_{\mathbb{F}}\left(\frac{D}{4^{r_2}\pi^d}\right)^{\frac{1}{2}}\Gamma^{r_1}\left(\frac{1}{2}\right)\Gamma^{r_2}(1)\lim_{z \rightarrow 1}\zeta_{\mathbb{F}}(z)\left[x^{\frac{z-1}{2}} - x^{\frac{1-z}{2}}\right] \notag\\
&= 2^{r_1+r_2}C_{\mathbb{F}}\sqrt{D}\lim_{z \rightarrow 1}\zeta_{\mathbb{F}}(z)\left[x^{\frac{z-1}{2}} - x^{\frac{1-z}{2}}\right]. \label{limit before series}
\end{align}}
To proceed further, we need the series expansion
\begin{align}\label{exp bracket around 1}
x^{\frac{z-1}{2}} - x^{\frac{1-z}{2}} &= \log(x)(z-1) + O((z-1)^2).
\end{align}
Hence, using \eqref{zeta around 1} and \eqref{exp bracket around 1} in \eqref{limit before series}, we have
\begin{align}
\lim_{z \rightarrow 1}2^d\pi^{r_2}C_{\mathbb{F}}\Omega_{\mathbb{F}}(z)\left[x^{\frac{z-1}{2}} - x^{\frac{1-z}{2}}\right] = 2^{r_1+r_2}\sqrt{D}C_{\mathbb{F}}H_{\mathbb{F}}\log(x). \label{limit value at 1}
\end{align}
Thus, utilizing \eqref{limit value at 1} and the fact that $\sigma_{\mathbb{F}, 1}(n) = n\sigma_{\mathbb{F}, -1}(n)$, in \eqref{IF1}, one gets
\begin{align}
\mathcal{I}_{\mathbb{F}, 1}(x) &= 2\sqrt{\frac{D}{\pi^d}}\sum_{n=1}^{\infty}\frac{\sigma_{\mathbb{F}, 1}(n)}{n}V_{\mathbb{F}, 1}\left(\left(\frac{\pi^{d} n x}{D}\right)^2\right) + 2^{r_1+r_2}\sqrt{D}C_{\mathbb{F}}H_{\mathbb{F}}\log(x). \label{IF1(x)}
\end{align}
We now find the value of $\mathcal{I}_{\mathbb{F}, -1}\left(\frac{1}{x}\right)$ as follows:
\begin{align}
\mathcal{I}_{\mathbb{F}, -1}\left(\frac{1}{x}\right) &= \frac{2}{x}\sqrt{\frac{\pi^d}{D}}\sum_{n=1}^{\infty}\sigma_{\mathbb{F}, 1}(n)V_{\mathbb{F}, -1}\left(\left(\frac{\pi^{d} n}{Dx}\right)^2\right) + 2^d\pi^{r_2}C_{\mathbb{F}}\Omega_{\mathbb{F}}(-1)\left(x - \frac{1}{x}\right) \notag \\
&= 2\sqrt{\frac{D}{\pi^d}}\sum_{n=1}^{\infty}\frac{\sigma_{\mathbb{F}, 1}(n)}{n}V_{\mathbb{F}, 1}\left(\left(\frac{\pi^{d} n}{Dx}\right)^2 \right) + \frac{2^{r_1}D}{\pi^{r_1+r_2}}C_{\mathbb{F}}\zeta_{\mathbb{F}}(2)\left(x - \frac{1}{x}\right), \label{IF1(1/x) final}
\end{align}
where in the last step, we used \eqref{Dedekind_zeta_functional_equation}, \eqref{Omega_F of fnal eqn} and \eqref{V_z and V_-z}.
From \eqref{IF1(x)} and \eqref{IF1(1/x) final}, we have
\begin{align}
&2\sqrt{\frac{D}{\pi^d}}\sum_{n=1}^{\infty}\frac{\sigma_{\mathbb{F}, 1}(n)}{n}V_{\mathbb{F}, 1}\left(\left(\frac{\pi^{d} n x}{D}\right)^2\right) + 2^{r_1+r_2}\sqrt{D}C_{\mathbb{F}}H_{\mathbb{F}}\log(x) - \frac{2^{r_1}D}{\pi^{r_1+r_2}}C_{\mathbb{F}}\zeta_{\mathbb{F}}(2)x \nonumber \\
&= 2\sqrt{\frac{D}{\pi^d}}\sum_{n=1}^{\infty}\frac{\sigma_{\mathbb{F}, 1}(n)}{n}V_{\mathbb{F}, 1}\left(\left(\frac{\pi^{d} n}{Dx}\right)^2 \right) - \frac{2^{r_1}D}{\pi^{r_1+r_2}x}C_{\mathbb{F}}\zeta_{\mathbb{F}}(2). \label{Before solving Steen}
\end{align}
Next, we further simplify the Steen function in the series using \eqref{Short notation}, as follows:
\begin{align}
V_{\mathbb{F}, 1}&\left(\left(\frac{\pi^{d} n x}{D}\right)^2\right) = V\left(\left(\frac{\pi^{d} n x}{D}\right)^2\bigg|\left(0\right)_{r_1 + r_2}, \left(\frac{1}{2}\right)_{r_1 + r_2}, \left(\frac{1}{2}\right)_{r_2}, (1)_{r_2}\right) \nonumber \\
&= \frac{1}{2\pi i}\int\limits_{(c)}\left[\Gamma(s)\Gamma\left(s+\frac{1}{2}\right)\right]^{r_1+r_2}\!\left[\Gamma\left(s+\frac{1}{2}\right)\Gamma\left(s+1\right)\right]^{r_2}\!\left(\frac{\pi^{d} n x}{D}\right)^{-2s} \textrm{d}s \nonumber \\
&= \frac{2^{r_1+r_2}\pi^{\frac{d}{2}}}{2\pi i}\int\limits_{(c)}\Gamma^{r_1+r_2}(2s)\Gamma^{r_2}(2s+1)\left(\frac{(2\pi)^{d} n x}{D}\right)^{-2s} \textrm{d}s \nonumber \\
&= \frac{2^{r_1+r_2-1}\pi^{\frac{d}{2}}}{2\pi i}\int\limits_{(2c)}\Gamma^{r_1+r_2}(s)\Gamma^{r_2}(s+1)\left(\frac{(2\pi)^{d} n x}{D}\right)^{-s} \textrm{d}s \nonumber \\
&= 2^{r_1+r_2-1}\pi^{\frac{d}{2}}V\left(\left(\frac{(2\pi)^{d} n x}{D}\right)\bigg|\left(0\right)_{r_1 + r_2}, (1)_{r_2}\right), \quad |\Arg(x)|<\frac{\pi d}{2}.
\end{align}
Substituting $\sqrt{D}H_{\mathbb{F}} = -2^{r_1}(2\pi)^{r_2}C_{\mathbb{F}}$ and the above value in \eqref{Before solving Steen} we get the final result \eqref{I_F(1,x)}.
\end{proof}


\begin{proof}[Corollary \rm{\ref{Case z=1 Q}}][]
When $\mathbb{F} = \mathbb{Q}$, we have $D = 1, r_1 = 1, r_2 = 0, C_\mathbb{Q} = -\frac{1}{2}, H_\mathbb{Q} = 1$ and $\zeta_\mathbb{Q}(2) = \zeta(2) = \frac{\pi^2}{6}$. Using these values and \eqref{Bessel fn formula} in \eqref{IFx}, we get
\begin{align*}
\mathcal{I}_{\mathbb{Q}, 1}(x) &= 4\sqrt{x}\sum_{n=1}^{\infty}\frac{\sigma(n)}{\sqrt{n}}K_{\frac{1}{2}}(2\pi nx) - \frac{\log(x)}{2} + \frac{\pi x}{6} \\
&= 2\sum_{n=1}^{\infty}\frac{\sigma(n)}{n}e^{-2\pi nx} - \frac{\log(x)}{2} + \frac{\pi x}{6},
\end{align*}
where in the last step, we have used that $K_{\frac{1}{2}}(w) = \sqrt{\frac{\pi}{2w}}e^{-w}$. Employing the above equation in Theorem \ref{Case z=1 F}, we have the desired result.
\end{proof}

\begin{proof}[Theorem \rm{\ref{Dedekind at 2 thm}}][]
The proof of this result immediately follows from Theorem \ref{Case z=1 F} by defining
\begin{align*}
S_\mathbb{F}(x) = \frac{2^{r_2}\pi^{r_1+r_2}}{\sqrt{D}C_\mathbb{F}}\sum_{n=1}^{\infty}\frac{\sigma_{\mathbb{F}, 1}(n)}{n}V\left(\left(\frac{(2\pi)^{d} n x}{D}\right)\bigg|\left(0\right)_{r_1 + r_2}, (1)_{r_2}\right),
\end{align*}
and rearranging the terms in \eqref{I_F(1,x)}.
\end{proof}

\begin{proof}[Theorem \rm{\ref{Dedekind zeta at 2m+1}}][]
We substitute $z=2m+1$ in Theorem \ref{Main thm 1} and use functional equation \eqref{Dedekind_zeta_functional_equation} of the Dedekind zeta function to have
\begin{align*}
&\left(\frac{\pi^d}{D}\right)^{-m-\frac{1}{2}}\!\!2x^{-m}\sum_{n=1}^{\infty}\sigma_{\mathbb{F}, -2m-1}(n)V_{\mathbb{F}, 2m+1}\left(\left(\frac{\pi^d n x}{D}\right)^2\right) + 2^d\pi^{r_2}C_{\mathbb{F}}\Omega_{\mathbb{F}}(2m+1)\left[x^{m} - x^{-m}\right] \\
&= \left(\frac{\pi^d}{D}\right)^{m+\frac{1}{2}}\!\!2x^{-m-1}\sum_{n=1}^{\infty}\sigma_{\mathbb{F}, 2m+1}(n)V_{\mathbb{F}, -2m-1}\left(\left(\frac{\pi^d n}{Dx}\right)^2\right) + 2^d\pi^{r_2}C_{\mathbb{F}}\Omega_{\mathbb{F}}(2m+2)\left[x^{m+1} - x^{-m-1}\right].
\end{align*}
We now use the relations \eqref{Gen sigma fn} and \eqref{V_z and V_-z} on the right hand side of the last equation to have
\begin{align*}
&\left(\frac{\pi^d}{D}\right)^{-m-\frac{1}{2}}\!\!2x^{-m}\sum_{n=1}^{\infty}\sigma_{\mathbb{F}, -2m-1}(n)V_{\mathbb{F}, 2m+1}\left(\left(\frac{\pi^d n x}{D}\right)^2\right) + 2^d\pi^{r_2}C_{\mathbb{F}}\Omega_{\mathbb{F}}(2m+1)\left[x^{m} - x^{-m}\right] \\
&= \left(\frac{\pi^d}{D}\right)^{-m-\frac{1}{2}}\!\!2x^m\sum_{n=1}^{\infty}\sigma_{\mathbb{F}, -2m-1}(n)V_{\mathbb{F}, 2m+1}\left(\left(\frac{\pi^d n }{Dx}\right)^2\right) + 2^d\pi^{r_2}C_{\mathbb{F}}\Omega_{\mathbb{F}}(2m+2)\left[x^{m+1} - x^{-m-1}\right].
\end{align*}
Therefore, we arrive at
\begin{align*}
&\Omega_{\mathbb{F}}(2m+1)\left[x^{m} - x^{-m}\right] = K_{\mathbb{F}, m}(x) - K_{\mathbb{F}, m}\left(\frac{1}{x}\right) + \Omega_{\mathbb{F}}(2m+2)\left[x^{m+1} - x^{-m-1}\right],
\end{align*}
where
\begin{align*}
K_{\mathbb{F}, m}(x) = \left(\frac{D}{\pi^d}\right)^{m+\frac{1}{2}}\frac{x^m}{2^{d-1}\pi^{r_2}C_{\mathbb{F}}}\sum_{n=1}^{\infty}\sigma_{\mathbb{F}, -2m-1}(n)V_{\mathbb{F}, 2m+1}\left(\left(\frac{\pi^d n}{Dx}\right)^2\right).
\end{align*}
Hence, using the value \eqref{Omega_F of fnal eqn} of $\Omega_\mathbb{F}(z)$  in the above equation, we get
\begin{align}
\zeta_{\mathbb{F}}(2m+1)  \left[x^{m} - x^{-m}\right] & =  \left( \frac{4^{r_2} \pi^d}{D}\right)^{m+\frac{1}{2}} \frac{1}{\Gamma^{r_1}\left( m+\frac{1}{2} \right) {(2m)!}^{r_2}}
\left[K_{\mathbb{F}, m}(x) - K_{\mathbb{F}, m}\left(\frac{1}{x}\right)\right] \notag \\
+ \sqrt{\frac{D}{4^{r_2}\pi^d}}&\frac{(m!)^{r_1} (2m+1)^{r_2} }{ \Gamma^{r_1}\!\left(m+\frac{1}{2}\right)} \zeta_{\mathbb{F}}(2m+2) \left[x^{m+1} - x^{-m-1}\right]. \label{Zeta_F(2m+1)(x -1/x)}
\end{align}
We finally use the duplication formula \eqref{Legendre's Duplication Formula} for the value of $\Gamma^{r_1}\!\left(m+\frac{1}{2}\right)$ in \eqref{Zeta_F(2m+1)(x -1/x)} to obtain
\begin{align*}
&\zeta_{\mathbb{F}}(2m+1)\left[x^{m} - x^{-m}\right] \\
&= \left( \frac{4^{r_2} \pi^d}{D}\right)^{m+\frac{1}{2}} \left( \frac{(m-1)!}{ 2^{1-2m} \sqrt{\pi} (2m-1)!} \right)^{r_1} \left(\frac{1}{(2m)!} \right)^{r_2}\left[K_{\mathbb{F}, m}(x) - K_{\mathbb{F}, m}\left(\frac{1}{x}\right)\right] \notag \\
&+ \sqrt{\frac{D}{4^{r_2}\pi^d}} \left(\frac{m! (m-1)!}{2^{1-2m} \sqrt{\pi} (2m-1)!} \right)^{r_1} (2m+1)^{r_2} \zeta_{\mathbb{F}}(2m+2) \left[x^{m+1} - x^{-m-1}\right].
\end{align*}
This completes the proof of Theorem \ref{Dedekind zeta at 2m+1}.
\end{proof}

For all the number fields listed in the table below, the class number $h_{\mathbb{F}}$ is 1. The values of $\zeta_{\mathbb{F}}(2)$ were computed using \eqref{Dedekind zeta at 2} with $x=2$. All numerical computations were carried out using \emph{Mathematica} software.
\begin{table}[h]
\centering
\renewcommand{\arraystretch}{1.4}
\caption{Values of $\zeta_{\mathbb{F}}(2)$ for different fields} \label{Zeta_2}
\vspace{0.2cm}
\begin{tabular}{|p{1.45cm}|p{2cm}|p{1.6cm}|p{1.8cm}|p{1.2cm}|p{2cm}|}
\hline
\textbf{Field ($\mathbb{F}$)} & \textbf{Regulator} ($R_\mathbb{F}$) & $\#$\textbf{(Roots of unity)} ($w_\mathbb{F}$) & \textbf{$C_\mathbb{F}$} &\textbf{Disc.} (D) & $\zeta_{\mathbb{F}}(2)$ (From \eqref{Dedekind zeta at 2}) \\
\hline
$\mathbb{Q}(\sqrt{2})$ & $\log(1 + \sqrt{2})$ & 2 & $-\frac{\log(1+\sqrt{2})}{2}$ & 8 & 1.43497 \\
\hline
$\mathbb{Q}(\sqrt{5})$ & $\log\left(\frac{1 + \sqrt{5}}{2}\right)$ & 2 & $-\frac{\log\left(\frac{1 + \sqrt{5}}{2}\right)}{2}$ & 5 & 1.16167 \\
\hline
$\mathbb{Q}(\sqrt{3})$ & $\log(2 + \sqrt{3})$ & 2 & $-\frac{\log(2 + \sqrt{3})}{2}$ & 12 & 1.5622 \\
\hline
$\mathbb{Q}(i)$  & 1 & 4 & $-\frac{1}{4}$ & $-4$ & 1.5067 \\
\hline
$\mathbb{Q}(\sqrt{-3})$ & 1 & 6 & $-\frac{1}{6}$ & $-3$ & 1.28519 \\
\hline
\end{tabular}
\end{table}

\subsection{Concluding Remarks}
Ramanujan recorded several remarkable identities in his Lost Notebook \cite{Ramanujan}, one of which, appearing on page 253, was later rediscovered by Guinand and is now known as the Ramanujan--Guinand identity \eqref{RG id}. This identity admits the Ramanujan--Koshliakov identity \eqref{Koshliakov Formula} and the transformation formula \eqref{Id z=1 Q} for the logarithm of the Dedekind eta function as special cases. In this paper, we established a number field analogue of the Ramanujan--Guinand identity for the Dedekind zeta function, (see Theorem \ref{Main thm 1}). As applications, we obtained number field analogues of the Ramanujan--Koshliakov identity and the transformation formula for the logarithm of the Dedekind eta function, with the classical identities recovered as the special cases. Our results also provide another illustration of the close relationship between transformation formulas and explicit evaluations of special values of zeta functions. In particular, as an application of our main theorem, we obtained a new formula for $\zeta_{\mathbb{F}}(2)$ for an arbitrary number field $\mathbb{F}$. This complements Zagier's celebrated evaluation, which expresses $\zeta_{\mathbb{F}}(2)$ in terms of special values of a logarithmic integral, whereas our formula takes the form of a rapidly convergent series. Moreover, we derived an identity for $\zeta_{\mathbb{F}}(2m+1), m \in \mathbb{N}$, in terms of similar rapidly convergent series.

\textbf{Acknowledgements}
The authors thank Prof. Atul Dixit and Prof. Rahul Kumar for reading the manuscript and giving their valuable suggestions. The first author acknowledges support from the Prime Minister Research Fellowship (PMRF), Government of India, under Grant No. 2102227. The last author is grateful for the support from the Anusandhan National Research Foundation (ANRF), India, through the ARG-MATRICS Grant (Grant No. ANRF/ARGM/2025/002002/MTR) and the Core Research Grant (Grant No. CRG/2023/002122). The authors also thank IIT Indore for its supportive research environment.

\textbf{Conflict of Interest}

The authors declare that they have no conflict of interest.

\textbf{Data Availability Statement}

No datasets were generated or analyzed during the current study.


\begin{thebibliography}{99}
\bibitem{BCH21}
S.~Banerjee, K.~Chakraborty and A.~Hoque, \emph{An analogue of Wilton’s formula and values of Dedekind zeta functions}, J.
Math. Anal. Appl., {\bf 495} (2021), 124675.

\bibitem{BGK23}
S.~Banerjee, R.~Gupta, and R.~Kumar,  \emph{A note on odd zeta values over any number field and Extended Eisenstein series},   J.  Math.   Anal.  Appl.,   {\bf 531}  (2024),  127883.

\bibitem{BK24}
D.~Banerjee and Khyati, \emph{Character analogues of Cohen-type identities and related Voronoi summation formulas}, Adv. Appl. Math., {\bf 153} (2024), pp.~10260.

\bibitem{BM23}
D.~Banerjee and B.~Maji, \emph{Identities associated to a generalized divisor function and modified Bessel function}, Res. Number Theory, {\bf 9}(28) (2023).

\bibitem{BM} 
D.~R.~Bansal and B.~Maji, \emph{A number field analogue of Ramanujan’s identity for $\zeta(2m+1)$}, J. Math. Anal. Appl., {\bf 550} (2025), 129538.

\bibitem{BM2}
D.~R.~Bansal and B.~Maji, \emph{Number field analogue of Jacobi theta relation and zeros of Dedekind zeta function on $\Re(s)= 1/2$},  (2025), arXiv:2507.18121.

\bibitem{BDS}
B.~C.~Berndt, A. Dixit,  and J. Sohn, \emph{Character analogues of theorems of Ramanujan, Koshliakov and Guinand}, Adv. Appl. Math., {\bf 46} (2011), 54--70.

\bibitem{BLS}
B.~C. Berndt, Y.~Lee, and J.~Sohn, \emph{In: Alladi, K. (ed.) The formulas of Koshliakov and Guinand in Ramanujan’s lost notebook, Surveys in Number Theory, Series:} Developments in Mathematics, vol. 17
(2008), Springer, New York, 21--42.  

\bibitem{CN1963} 
K.~Chandrasekharan and R.~Narasimhan,  \emph {The approximate functional equation for a class of zeta-functions}, Math. Annalen, {\bf 152} (1963), 30--64.

\bibitem{Cohen}
H.~Cohen, \emph{Some formulas of Ramanujan involving Bessel functions}, Publ. Math. Besan\c{c}on Alg\'ebre Th\'eorie Nr., (2010), 59--68.

\bibitem{Dixit} 
A.~Dixit, \emph{Transformation formulas associated with integrals involving the Riemann $\Xi$-function}, Monatsh.
Math., {\bf 164} No. 2 (2011), 133–156.

\bibitem{DKM18} 
A.~Dixit, A.~Kesarwani and V.~H.~Moll, \emph{A generalized modified Bessel function and a higher level analogue of the theta transformation formula (with an appendix by N. M. Temme)}, J. Math. Anal. Appl., {\bf 459} (2018), 385--418.

\bibitem{Ferrar1936}
W.~L.~Ferrar, \emph{Some solutions of the equation $F(t) = F(t^{-1})$}, J. Lond. Math. Soc., {\bf 11} (1936), 99--103.

\bibitem{Guinand}
A.~P. Guinand, \emph{Some rapidly convergent series for the Riemann $\xi$-function}, Quart. J. Math. (Oxford), {\bf 6} (1955), 156--160.



\bibitem{KT} 
S.~Kanemitsu and H.~Tsukada,  \emph{Contributions to the theory of zeta-functions},  The Modular Relation Supremacy, Series on Number Theory and Its Applications, vol.10, World Sci., Singapore, 2015.

\bibitem{Klingen}
H.~Klingen, \emph{\"{U}ber die Werte der Dedekindschen Zeta funktionen}, Math. Ann. \textbf{145} (1962), 265--272.

\bibitem{Koshliakov}
N.~S.~Koshliakov, \emph{On Voronoi’s sum-formula}, Mess. Math, {\bf 58} (1929), 30--32.

\bibitem{Kumar} 
R.~Kumar, \emph{Extensions of Watson’s theorem and the Ramanujan–Guinand formula}, Int. J. Number Theory, {\bf 19} (01) (2023), 199--222.

\bibitem{Luke} 
Y.~L.~Luke, The Special Functions and Their Approximations, vol. 1, UK Edition, Academic Press, Inc., 1969.

\bibitem{OLBC2010}
F.~W.~J.~Olver, D.~W.~Lozier, R.~F.~Boisvert, C.~W.~Clark (Eds.), NIST Handbook of Mathematical Functions, Cambridge University Press, Cambridge, 2010.

\bibitem{Ramanujan}
S.~Ramanujan, \emph{The Lost Notebook and Other Unpublished Papers}, Narosa, New Delhi, 1988.

\bibitem{Siegel}
C.~L.~Siegel, \emph{Berechnung von Zetafunktionen an ganzzahligen Stellen}, Nachr. Akad. Wiss. G\"{o}tt., \textbf{10}  (1969) 87--102.


\bibitem{Steen}
S.~W.~P.~Steen, \emph{Divisor functions: their differential equations and recurrence formulas}, Proc. Lond. Math. Soc. {\bf 2} (1930), 47--80.

\bibitem{Watson}
G.~N.~Watson, \emph{Some self-reciprocal functions}, Quart. J. Math. (Oxford) {\bf 2} (1931), 298--309.

\bibitem{Zagier}
D.~Zagier, \emph{Hyperbolic manifolds and special values of Dedekind zeta-functions}, Invent. Math. {\bf 83} (1986), 285--301.
\end{thebibliography}
\end{document}